\documentclass[11pt]{article}
\usepackage[letterpaper,margin=1in]{geometry}
\usepackage{amsmath,amssymb,amsthm,mathtools}
\usepackage{hyperref}
\usepackage{tikz}
\usetikzlibrary{arrows.meta,calc}

\hypersetup{
  colorlinks=true,
  linkcolor=blue,
  citecolor=blue,
  urlcolor=blue
}

\newtheorem{theorem}{Theorem}[section]
\newtheorem{lemma}[theorem]{Lemma}
\newtheorem{corollary}[theorem]{Corollary}
\newtheorem{remark}[theorem]{Remark}
\newtheorem{proposition}[theorem]{Proposition}

\DeclareMathOperator{\pf}{pf}
\DeclareMathOperator{\PM}{PerfMatch}
\DeclareMathOperator{\per}{per}
\DeclareMathOperator{\Pf}{Pf}
\DeclareMathOperator{\sign}{sign}

\newcommand{\cM}{\mathcal M}
\newcommand{\pfmax}{\pf_{\max}}

\title{Exponential Lower Bounds for the Pfaffian Number of Graphs}
\author{Priyanshu Pant\\
Indian Institute of Technology Indore\\
\texttt{priyanshupant03@gmail.com}
\and
Ranveer Singh\\
Indian Institute of Technology Indore\\
\texttt{ranveer@iiti.ac.in}}
\date{}
\begin{document}

\maketitle

\begin{abstract}
Galluccio--Loebl and Tesler showed that the perfect-matching polynomial
of a graph embedded in an orientable surface of genus \(g\) can be written
as a linear combination of at most \(4^g\) Pfaffians. We show that, in
general, exponentially many Pfaffians are necessary. More precisely, for
every \(g\ge1\), there exists a graph of orientable genus at most \(g\)
whose perfect-matching polynomial requires at least \((8/3)^g\) Pfaffians
in any such linear representation. In particular, for every even integer
\(n\ge6\), there is a graph on \(n\) vertices with Pfaffian number at least
\((8/3)^{\lfloor n/6\rfloor}\). Moreover, the lower bound is witnessed
even by cubic bipartite matching-covered graphs. 
We prove this by showing that expressing the permanent of an \(n\times n\) matrix
of distinct variables as a linear combination of determinants obtained
by changing signs of its entries requires exponentially many determinants.
As a consequence, we improve a recent linear lower bound on the
Pfaffian number due to Junchaya, Miranda, and Lucchesi to an
exponential lower bound.
\end{abstract}

\section{Introduction}
All graphs considered in this paper are finite and simple, that is, they
have no loops or multiple edges. For a graph \(G\), we write \(V(G)\) and
\(E(G)\) for its vertex set and edge set, respectively. A \emph{perfect matching}
of \(G\) is a set \(M\) of pairwise disjoint edges such that each vertex of
\(G\) is incident with exactly one edge of \(M\). Counting perfect matchings
is a classical problem in graph theory and combinatorics, with
applications in chemistry, statistical physics, and quantum mechanics; see
\cite{LovaszPlummer1986, LucchesiMurty2024}.
Computationally, the problem is hard. Even for bipartite graphs, it is
equivalent to computing the permanent of a \(0\)-\(1\) matrix, which is
\(\#P\)-complete by Valiant's theorem~\cite{Valiant1979}.
For an \(n\times n\) matrix \(A=(a_{ij})\), the \emph{permanent} of \(A\) is
defined as
\[
    \per(A)=\sum_{\sigma\in S_n}\prod_{i=1}^n a_{i,\sigma(i)},
\]
where \(S_n\) is the set of all permutations of \(\{1,2,\ldots,n\}\), whereas the determinant is defined by
\[
    \det(A)=
    \sum_{\sigma\in S_n}
    \sign(\sigma)\prod_{i=1}^n a_{i,\sigma(i)},
\]
where \(\sign(\sigma)\) is \(1\) for even permutations and \(-1\) for odd
permutations. Although the permanent and determinant differ only in the signs
of the same monomials, their computational behavior is strikingly different. The determinant can be computed in \(O(n^{\omega})\) arithmetic operations
using fast matrix multiplication, where \(2\le \omega<2.373\)
\cite{AhoHopcroftUllman1974,AlmanWilliams2021}.
This contrast in computational complexity motivates Pólya's permanent problem~\cite{Polya1913}, which asks
when, for a given \(0\)-\(1\) matrix \(A\), one can change signs of the nonzero
entries to obtain a matrix \(B\) with the same zero pattern such that
\[
    \det(B)=\per(A).
\]
If such a signing is given, then the permanent of \(A\) can be computed in
polynomial time by evaluating the corresponding determinant.
Moreover, for \(0\)-\(1\) matrices, the
existence of such a signing can be recognized, and a suitable signing can
be found, in polynomial time.
This signing problem is equivalent to several combinatorial
and graph-theoretic problems, including the existence of Pfaffian orientations
for bipartite graphs and the even directed cycle problem
\cite{McCuaig2004,RobertsonSeymourThomas1999}.

The graph analogue of the permanent--determinant contrast is the contrast
between the perfect-matching polynomial and Pfaffians.  Let \(G\) be a graph with \(2n\) vertices, and
assign a variable \(x_e\) to each edge \(e\in E(G)\).  The \emph{perfect-matching polynomial} of \(G\) is defined by
\[
    \PM(G)
    =
    \sum_{M\in\cM(G)}
    \prod_{e\in M}x_e,
\]
where \(\cM(G)\) denotes the set of perfect matchings of \(G\). Setting all variables \(x_e=1\), one obtains the number of perfect matchings of \(G\). Thus \(\PM(G)\) is the unsigned sum of the monomials corresponding to perfect
matchings, just as the permanent is the unsigned sum over permutation terms.

In the Pfaffian expansion, an orientation of the graph assigns signs to perfect matchings. Let \(D\) be an \emph{orientation} of \(G\), that is, a
choice of direction for each edge. Fix an ordering \(1,2,\ldots,2n\) of the vertices of \(G\).  If
\(e=ij\in E(G)\), we write \(x_e\) and \(x_{ij}\) interchangeably for the
corresponding edge variable.
The \emph{weighted skew-adjacency matrix} of \(G\) with respect to \(D\), denoted by \(A_D\),
is the \(2n\times 2n\) matrix whose \((i,j)\)-entry is
\[
    (A_D)_{ij}
    =
    \begin{cases}
        x_{ij}, & \text{if }ij\in E(G)\text{ is oriented from }i\text{ to }j,\\
        -x_{ij}, & \text{if }ij\in E(G)\text{ is oriented from }j\text{ to }i,\\
        0, & \text{if }ij\notin E(G).
    \end{cases}
\]
Let
\(
    M:=\{e_1,e_2,\ldots,e_n\}
\)
be a perfect matching of \(G\), and suppose that in the orientation \(D\)
of \(G\), the tail and the head of the edge \(e_i\), for
\(1\le i\le n\), are \(u_i\) and \(v_i\), respectively.  The \emph{permutation of \(M\) in \(D\)}, denoted by
\(\pi_D(M)\), is
\[
    \pi_D(M)=
    \begin{pmatrix}
        1&2&3&4&\cdots&2n-1&2n\\
        u_1&v_1&u_2&v_2&\cdots&u_n&v_n
    \end{pmatrix}.
\]
The sign of \(M\) with respect to \(D\), denoted by \(\sign_D(M)\), is defined as the sign
of the permutation \(\pi_D(M)\). The value of \(\sign_D(M)\) is independent of the order in which the edges of
\(M\) are enumerated. The Pfaffian of the weighted skew-adjacency matrix
\(A_D\) is given by
\[
    \Pf(A_D)
    =
    \sum_{M\in\cM(G)}
    \sign_D(M)\prod_{e\in M}x_e .
\]
Thus the Pfaffian sums over the same perfect matchings as \(\PM(G)\), but
with signs depending on the orientation \(D\). Note that the same perfect matching may have different signs under different
orientations of \(G\).
More generally, let \(A=(a_{ij})\) be any skew-symmetric matrix of order
\(2n\). We associate with \(A\) an oriented weighted graph \(D_A\) on
vertex set \(\{1,\ldots,2n\}\) as follows. For every \(i<j\) with
\(a_{ij}\neq 0\), add the edge \(ij\), orient it from \(i\) to \(j\), and
assign it weight \(a_{ij}\). Then \(A\) is the weighted skew-adjacency
matrix associated with \(D_A\). The Pfaffian of \(A\), denoted by
\(\Pf(A)\), is the Pfaffian of this weighted skew-adjacency matrix.
Cayley's identity~\cite{Cayley1849,Halton1966} states that every
skew-symmetric matrix \(A\) of even order satisfies
\[
    \Pf(A)^2=\det(A).
\]
Thus Pfaffians, like determinants, are computable in polynomial time.

If there exists an orientation \(D\) in which all perfect matchings of \(G\)
have the same sign, then \(\Pf(A_D)\) equals \(\PM(G)\) up to an overall
sign, so \(\PM(G)\), and hence the number of perfect matchings of \(G\),
can be computed efficiently. Such an orientation is called a \emph{Pfaffian orientation}, and a graph admitting such an orientation is called \emph{Pfaffian}. Thus deciding whether a graph is Pfaffian is the graph-theoretic counterpart of Pólya's permanent problem. In both problems, one asks whether suitable signs can turn a signed expansion into the unsigned one.

In 1961, Kasteleyn introduced Pfaffian orientations in connection with problems from statistical physics, and later proved that every planar graph is
Pfaffian~\cite{Kasteleyn1961,Kasteleyn1967}. Together with the work of Fisher
and Temperley, this gave the Fisher--Kasteleyn--Temperley
(FKT) algorithm, which counts perfect matchings in planar graphs in polynomial
time by reducing the computation to a Pfaffian, and hence to a determinant
\cite{Fisher1961,TemperleyFisher1961}.
More generally, whenever a Pfaffian orientation is available, perfect
matchings can be counted in polynomial time. However, not every graph is
Pfaffian. The smallest non-Pfaffian graph is \(K_{3,3}\), meaning that no
orientation of \(K_{3,3}\) makes all perfect matchings have the same sign
\cite{Little1975,NorinLittleTeo2004}.

For bipartite graphs, \(K_{3,3}\) plays a stronger structural role.
Little's theorem states that a bipartite graph is Pfaffian if and only if
it does not contain \(K_{3,3}\) as a matching minor~\cite{Little1975}. Thus \(K_{3,3}\) is not
merely the smallest non-Pfaffian example, but the obstruction to
Pfaffianity in the bipartite setting. This viewpoint has more recently led
to a broader structural theory of matching minors
\cite{GiannopoulouThilikosWiederrecht2023,
      GiannopoulouWiederrecht2024}.
In bipartite graphs, Pfaffian orientations can be recognized and found in
polynomial time by Robertson, Seymour, and Thomas, and independently by
McCuaig
\cite{McCuaig2004,RobertsonSeymourThomas1999}.
For general graphs, the complexity of deciding whether a Pfaffian orientation
exists remains open. Vazirani and Yannakakis showed that several algorithmic
questions around Pfaffian orientations, \(0\)-\(1\) permanents, and even
directed cycles are polynomial-time equivalent
\cite{VaziraniYannakakis1988}.

 For non-Pfaffian graphs, one can ask whether \(\PM(G)\) can still be
obtained as a linear combination of several Pfaffians.
Following Norine~\cite{Norine2009}, a graph \(G\) is called
\emph{\(k\)-Pfaffian} if there exist orientations
\(D_1,\ldots,D_k\) of \(G\) and real coefficients
\(c_1,\ldots,c_k\) such that
\[
    \PM(G)=\sum_{i=1}^k c_i\Pf(A_{D_i}).
\]
The \emph{Pfaffian number} of \(G\), denoted \(\pf(G)\), is the least positive
integer \(k\) such that \(G\) is \(k\)-Pfaffian. Thus \(\pf(G)=1\) exactly
when \(G\) is Pfaffian.
Norine proved that every \(3\)-Pfaffian graph is Pfaffian and every
\(5\)-Pfaffian graph is \(4\)-Pfaffian~\cite{Norine2009}. Motivated by these
results, Norine conjectured that Pfaffian numbers are powers of \(4\).
Miranda and Lucchesi disproved this conjecture by constructing a graph with
Pfaffian number \(6\)~\cite{MirandaLucchesi2011}. They further conjectured
that every Pfaffian number greater than \(1\) is even, and that every even
integer \(k\ge4\) occurs as the Pfaffian number of some graph. This conjecture remains open. 

The Pfaffian method also extends to graphs on surfaces. The orientable genus of a graph \(G\), denoted \(\gamma(G)\), is the smallest
integer \(g\) such that \(G\) can be embedded without crossings on an
orientable surface with \(g\) handles. Thus planar graphs are precisely
the graphs with orientable genus \(0\). Kasteleyn stated, without proof, that for every graph of orientable genus at
most \(g\), the perfect-matching polynomial can be written as a linear
combination of \(4^g\) Pfaffians. 
This statement was later proved independently by
Galluccio and Loebl~\cite{GalluccioLoebl1999} and by
Tesler~\cite{Tesler2000}. Tesler also extended the result to non-orientable
surfaces.
Thus orientable genus gives a general upper bound on the number of Pfaffians
needed to represent \(\PM(G)\), that is,
\[
    \pf(G)\le 4^g \text{ whenever }\gamma(G)\le g.
\]

A \(k\)-Pfaffian representation is an algebraic representation of the
perfect-matching polynomial.  Once the orientations and coefficients are
available, it allows \(\PM(G)\) to be evaluated using \(k\) Pfaffian
computations.
For graphs of orientable genus \(g\), the theorem of Galluccio--Loebl and
Tesler gives an explicit representation using \(4^g\) Pfaffians. A natural
question is whether an exponential number of Pfaffian terms is necessary.

A related but different question was studied by Curticapean and Marx
\cite{CurticapeanMarx2016}. They showed that, assuming \(\#\)ETH, counting
perfect matchings on graphs of genus \(g\) admits no
$$
    2^{o(g)}n^{O(1)}
$$
time algorithm. This does not imply an exponential lower bound on the
Pfaffian number, since a short Pfaffian representation could in principle
exist without being efficiently constructible. Here we ask whether
exponentially many Pfaffian terms are necessary even at the level of
existence.
More recently, Junchaya, Miranda, and Lucchesi initiated the study of lower
bounds for the Pfaffian number. In particular, they proved
\[
    \pf(K_{3r,3r})\ge 3r+1,
\]
and hence showed that Pfaffian numbers are unbounded
\cite{JunchayaMirandaLucchesi2026}.

In this paper we show that the Pfaffian number can grow
exponentially with the orientable genus. Let \(\pfmax(g)\) denote the maximum Pfaffian number
among all graphs of orientable genus at most \(g\), that is,
\[
    \pfmax(g):=\max\{\pf(G):\gamma(G)\le g\}.
\]

\begin{theorem}
\label{thm:intro-genus-main}
For every \(g\ge1\),
\[
    \pfmax(g)\ge \left(\frac83\right)^g.
\]
\end{theorem}

 In fact, for every \(g\ge1\), we construct a cubic bipartite matching-covered graph
\(G_g\) of orientable genus exactly \(g\) such that
\[
    \pf(G_g)\ge
    \left\lceil\left(\frac83\right)^g\right\rceil .
\]
Combining our lower bound with the theorem of Galluccio--Loebl and Tesler,
we obtain
\[
    \left(\frac83\right)^g\le \pfmax(g)\le 4^g.
\]
This shows that the \(4^g\) upper bound cannot, in general, be replaced by a subexponential bound in \(g\).  In other words, there are graphs of
orientable genus \(g\) for which every representation of \(\PM(G)\) as a
linear combination of Pfaffians requires exponentially many terms. 
At the heart of the proof is the following matrix-theoretic lower bound.

\begin{theorem}
\label{thm:permanent-signed-det}
Let \(A=(a_{ij})\) be an \(n\times n\) matrix of distinct variables. Suppose
that \(\per(A)\) can be written as a linear combination
\[
    \per(A)=\sum_{i=1}^k c_i\det(A_i),
\]
where each \(c_i\in\mathbb R\), and each \(A_i\) is obtained from \(A\) by
changing signs of entries independently. Then
\[
    k\ge
    \left\lceil
    \left(\frac83\right)^{\lfloor n/3\rfloor}
    \right\rceil .
\]
\end{theorem}

Thus, as a polynomial identity, writing the permanent as a linear combination
of determinants obtained by changing signs of entries requires exponentially
many determinants. This lower bound applies to this signed-determinant
representation model, rather than to general determinantal representations of
the permanent, and may be of independent interest.

We now apply this matrix-theoretic lower bound to Pfaffian numbers of graphs.
It gives a structural consequence related to
Little's \(K_{3,3}\) obstruction. If a graph \(G\) contains a conformal
subgraph that is a disjoint union of \(r\) bisubdivisions of \(K_{3,3}\),
then
\[
    \pf(G)\ge
    \left\lceil\left(\frac83\right)^r\right\rceil .
\]
Thus a conformal subgraph consisting of \(r\) disjoint \(K_{3,3}\)
obstructions already forces the Pfaffian number to grow exponentially.
For complete bipartite graphs, the matrix lower bound gives the following
exponential bound.

\begin{theorem}
\label{thm:intro-complete-bipartite}
For every \(n\ge3\),
\[
    \pf(K_{n,n})
    \ge
    \left\lceil
    \left(\frac83\right)^{\lfloor n/3\rfloor}
    \right\rceil .
\]
\end{theorem}

For \(n=3r\), this gives an exponential lower bound in \(r\), improving asymptotically on the previously known linear lower bound
\[
    \pf(K_{3r,3r})\ge 3r+1
\]
of Junchaya, Miranda, and Lucchesi. 
They also obtained
\[
    \pf(K_{2n})=\Omega(n\log n)
\]
and suggested that the Pfaffian number might
grow exponentially for complete graphs~\cite{JunchayaMirandaLucchesi2026}. 
Since \(K_{2n}\) contains a spanning copy of \(K_{n,n}\), the same lower
bound also holds for even complete graphs.

\begin{theorem}
\label{thm:intro-complete-graphs}
For every \(n\ge3\),
\[
    \pf(K_{2n})
    \ge
    \left\lceil
    \left(\frac83\right)^{\lfloor n/3\rfloor}
    \right\rceil .
\]
\end{theorem}

The rest of the paper is organized as follows.
Section~\ref{sec:prelim} records the preliminary reductions used in the proofs.
Section~\ref{sec:lower-bound} proves the matrix lower bound for representing
the permanent as a linear combination of signed determinants.
Section~\ref{sec:consequences} first derives the main genus lower bound
from disjoint copies of \(K_{3,3}\), then strengthens it to cubic
bipartite matching-covered graphs of the same genus, and finally derives
the bounds for complete bipartite and complete graphs.

\section{Preliminaries}\label{sec:prelim}

In this section, we explain how, for bipartite
graphs, Pfaffians arising from orientations can be written as signed
determinants of weighted biadjacency matrices. We also record a simple
monotonicity property of the
Pfaffian number under conformal subgraphs. We write \(\sqcup\) for disjoint union, both for sets and
for graphs. We write \(\circ\) for entrywise product of matrices of the
same size.

Let \(G=(U\sqcup V,E)\) be a bipartite graph on \(2n\) vertices, with
\(|U|=|V|=n\). The \emph{weighted biadjacency matrix} of \(G\), denoted
by \(B_G\), is the matrix with rows indexed by \(U\), columns indexed by
\(V\), and entries
\[
    (B_G)_{uv}
    =
    \begin{cases}
        x_{uv}, & \text{if }uv\in E(G),\\
        0, & \text{otherwise}.
    \end{cases}
\]
Then \(\PM(G)=\per(B_G)\), since choosing one nonzero entry from each row
and each column of \(B_G\) is the same as choosing a perfect matching of
\(G\).

\begin{lemma}
\label{lem:bipartite-pf-det}
Let \(G=(U\sqcup V,E)\) be a bipartite graph with \(|U|=|V|=n\). Order the vertices with \(U\) first and \(V\) second.
For every orientation \(D\) of \(G\), there is a sign matrix
\(S_D\in\{\pm1\}^{n\times n}\) such that
\[
    \Pf(A_D)
    =
    (-1)^{n(n-1)/2}\det(S_D\circ B_G).
\]
\end{lemma}

\begin{proof}
Write
\(
    U=\{u_1,\ldots,u_n\},
\) and \(
    V=\{v_1,\ldots,v_n\},
\)
and order the vertices as
\(
    u_1,\ldots,u_n,v_1,\ldots,v_n.
\)
With this vertex order, the weighted skew-adjacency matrix of \(D\) has
block form
\[
    A_D=
    \begin{pmatrix}
        0 & S_D\circ B_G\\
        -(S_D\circ B_G)^T & 0
    \end{pmatrix}
\]
for some sign matrix \(S_D\in\{\pm1\}^{n\times n}\). Define \(S_D=(s_{ij})\in\{\pm1\}^{n\times n}\) by
\[
    s_{ij}
    =
    \begin{cases}
        1, & \text{if } u_iv_j\in E(G)\text{ is oriented from }u_i\text{ to }v_j,\\
        -1, & \text{if } u_iv_j\in E(G)\text{ is oriented from }v_j\text{ to }u_i,\\
        1, & \text{if } u_iv_j\notin E(G).
    \end{cases}
\]
The value of \(s_{ij}\) on non-edges is irrelevant, because then
\((B_G)_{ij}=0\).
Set
\(
    C:=S_D\circ B_G.
\)
We show that
\[
    \Pf
    \begin{pmatrix}
        0 & C\\
        -C^T & 0
    \end{pmatrix}
    =
    (-1)^{n(n-1)/2}\det(C).
\]
In the Pfaffian expansion, a nonzero term cannot pair two vertices both in
\(U\), nor two vertices both in \(V\), because the two diagonal blocks are
zero.  Hence every nonzero term pairs each \(u_i\) with exactly one vertex
\(v_{\sigma(i)}\), where \(\sigma\in S_n\).  Thus the nonzero terms are
indexed by permutations \(\sigma\in S_n\).
For a fixed \(\sigma\), let
\(
    M_\sigma=\{u_1v_{\sigma(1)},\ldots,u_nv_{\sigma(n)}\}.
\)
The corresponding product of matrix entries is
\[
    \prod_{i=1}^n C_{i,\sigma(i)}.
\]
It remains to compute the sign of this term in the Pfaffian expansion. With our vertex
order, the pairing \(M_\sigma\) gives the ordered list
\[
    u_1,v_{\sigma(1)},u_2,v_{\sigma(2)},\ldots,u_n,v_{\sigma(n)}.
\]
Move all \(V\)-vertices to the right, keeping their relative order.  This
requires
\[
    (n-1)+(n-2)+\cdots+1+0=\frac{n(n-1)}2
\]
transpositions, and gives
\[
    u_1,u_2,\ldots,u_n,
    v_{\sigma(1)},v_{\sigma(2)},\ldots,v_{\sigma(n)}.
\]
The remaining order
\(
    v_{\sigma(1)},v_{\sigma(2)},\ldots,v_{\sigma(n)}
\)
has sign \(\sign(\sigma)\) relative to
\(
    v_1,v_2,\ldots,v_n.
\)
Therefore the Pfaffian sign of the term indexed by \(\sigma\) is
\(
    (-1)^{n(n-1)/2}\sign(\sigma).
\)
Hence
\[
    \Pf(A_D)
    =
    \sum_{\sigma\in S_n}
    (-1)^{n(n-1)/2}\sign(\sigma)
    \prod_{i=1}^n C_{i,\sigma(i)}.
\]
Factoring out the constant sign gives
\[
    \Pf(A_D)
    =
    (-1)^{n(n-1)/2}
    \sum_{\sigma\in S_n}
    \sign(\sigma)\prod_{i=1}^n C_{i,\sigma(i)}.
\]
The last sum is exactly \(\det(C)\).  Since \(C=S_D\circ B_G\), we obtain
\[
    \Pf(A_D)
    =
    (-1)^{n(n-1)/2}\det(S_D\circ B_G).
\]
This proves the lemma.
\end{proof}

\begin{lemma}
\label{lem:bipartite-signed-determinants}
Let \(G=(U\sqcup V,E)\) be a bipartite graph with \(|U|=|V|=n\).
If \(\pf(G)=k\), then there exist sign matrices
\(S_1,\ldots,S_k\in\{\pm1\}^{n\times n}\) and real coefficients
\(c_1,\ldots,c_k\) such that
\[
    \per(B_G)
    =
    \sum_{i=1}^k c_i\det(S_i\circ B_G).
\]
\end{lemma}

\begin{proof}
By definition of $\pf(G)$, there are orientations
\(D_1,\ldots,D_k\) of \(G\) and real coefficients
\(d_1,\ldots,d_k\) such that
\[
    \PM(G)
    =
    \sum_{i=1}^k d_i\Pf(A_{D_i}).
\]
By Lemma~\ref{lem:bipartite-pf-det}, for each \(i\) there is a sign
matrix \(S_i\in\{\pm1\}^{n\times n}\) such that
\[
    \Pf(A_{D_i})
    =
    (-1)^{n(n-1)/2}\det(S_i\circ B_G).
\]
Therefore
\[
    \PM(G)
    =
    \sum_{i=1}^k
    d_i(-1)^{n(n-1)/2}\det(S_i\circ B_G).
\]
Since \(G\) is bipartite,
\[
    \PM(G)=\per(B_G).
\]
Setting
\(
    c_i:=d_i(-1)^{n(n-1)/2}
\)
gives
\[
    \per(B_G)
    =
    \sum_{i=1}^k c_i\det(S_i\circ B_G). \qedhere
\]
\end{proof}

In a graph, edges that do not belong to any perfect matching do not occur in
the perfect-matching polynomial and may therefore be deleted without changing
\(\PM(G)\).
After deleting all such edges, every nontrivial connected
component is matching-covered. Recall that a graph is
\emph{matching-covered} if it is connected and every edge belongs to a
perfect matching. The perfect-matching polynomial then factors over these
components. Thus matching-covered graphs are the natural indecomposable
objects from the viewpoint of perfect matchings. We nevertheless allow general, and in particular disconnected, graphs
throughout the paper, since they will be useful as intermediate constructions
for proving lower bounds that can then be transferred to more structured
graphs.

A subgraph \(H\) of \(G\) is called \emph{conformal} if both \(H\) and
\(G-V(H)\) have perfect matchings, where the empty graph is regarded as
having a perfect matching. We will use the following monotonicity property,
which also follows from results in
\cite{JunchayaMirandaLucchesi2026}; we include a short direct proof.

\begin{proposition}
\label{prop:conformal}
If \(H\) is a conformal subgraph of \(G\), then
\[
    \pf(G)\ge \pf(H).
\]
\end{proposition}

\begin{proof}
Choose a vertex ordering so that the vertices of \(H\)
come first. Let \(\pf(G)=k\). Then there exist orientations \(D_1,\ldots,D_k\) of \(G\) and
real coefficients \(c_1,\ldots,c_k\) such that
\[
    \PM(G)=\sum_{i=1}^k c_i\Pf(A_{D_i}).
\]
Since \(H\) is conformal, \(G-V(H)\) has a perfect matching \(N\).
Set \(x_e=0\) for every
\(e\in E(G)\setminus\bigl(E(H)\cup N\bigr)\), and set \(x_e=1\) for
every \(e\in N\). After this substitution, every surviving perfect matching
of \(G\) must contain \(N\), and hence is of the form
\(M\cup N\), where \(M\in\mathcal M(H)\). Since the edge variables of
\(N\) have been set to \(1\), the left-hand side becomes
\(
    \PM(H).
\)

Let \(D_i|_H\) and \(D_i|_N\) denote the restrictions of \(D_i\) to
\(H\) and \(N\), respectively. After the same substitution, the weighted
skew-adjacency matrix becomes block diagonal, with blocks
\(A_{D_i|_H}\) and \(A_{D_i|_N}\). Hence
\[
    \Pf(A_{D_i})
    =
    \Pf(A_{D_i|_H})\Pf(A_{D_i|_N}).
\]
Here \(\Pf(A_{D_i|_N})\) is a constant, since all edge variables of
\(N\) have been set to \(1\).
Therefore
\[
    \PM(H)
    =
    \sum_{i=1}^k
    c_i\Pf(A_{D_i|_N})\Pf(A_{D_i|_H}).
\]
Absorbing \(\Pf(A_{D_i|_N})\) into the coefficient \(c_i\), we obtain a
Pfaffian representation of \(\PM(H)\) using at most \(k\) terms. Therefore
\[
    \pf(H)\le k=\pf(G). \qedhere
\]
\end{proof}

A subgraph \(H\) of \(G\) is called a \emph{spanning subgraph} if it has the
same vertex set as \(G\).
In particular, every spanning subgraph that has a perfect matching is
conformal, so Proposition~\ref{prop:conformal} applies to such spanning
subgraphs as well.

\section{A Matrix-Theoretic Lower Bound}\label{sec:lower-bound}

This section proves the matrix lower bound. We first record two elementary
facts about \(3\times3\) sign matrices, and then apply them to block
diagonal matrices.

\subsection{Permanent and determinant counts}

Let
\(
    \Sigma_3:=\{\pm1\}^{3\times3}
\)
be the set of all \(3\times3\) sign matrices. Thus
\(
    |\Sigma_3|=2^9=512.
\)
The two facts we need are that the permanent is nonzero
on every element of \(\Sigma_3\), while, for every fixed sign matrix \(S\),
the determinant \(\det(S\circ X)\) is nonzero for exactly \(192\) choices of
\(X\in\Sigma_3\).

\begin{lemma}
\label{lem:per-nonzero}
For every \(X\in\Sigma_3\),
\(
    \per(X)\neq0.
\)
\end{lemma}

\begin{proof}
Multiplying rows and columns of a matrix by $+1$ or $-1$ changes the permanent by an overall sign and hence preserves whether the
permanent is zero.
Every \(3\times3\) sign matrix can be normalized by row and column sign
changes to the form
\[
    X=
    \begin{pmatrix}
        1&1&1\\
        1&a&b\\
        1&c&d
    \end{pmatrix},
    \qquad a,b,c,d\in\{\pm1\}.
\]
For this matrix,
\[
    \per(X) =a+b+c+d+ad+bc.
\]
We check the four cases according to \((a,b)\):
\[
\begin{array}{c|c}
(a,b) & \per(X) \\
\hline
(1,1) & 2+2c+2d\\
(1,-1) & 2d\\
(-1,1) & 2c\\
(-1,-1) & -2.
\end{array}
\]
None of these expressions is zero for any \(c,d\in\{\pm1\}\).  Therefore every
matrix in \(\Sigma_3\) has nonzero permanent.
\end{proof}

\begin{lemma}
\label{lem:det-192}
For any fixed sign matrix \(S\in\{\pm1\}^{3\times3}\),
\(
    \det(S\circ X)\neq0
\)
for exactly \(192\) choices of \(X\in\Sigma_3\).
\end{lemma}

\begin{proof}
For any fixed \(S\in\{\pm1\}^{3\times3}\), the map
\(
    X\mapsto S\circ X
\)
is its own inverse on \(\Sigma_3\).  Hence it is a bijection of
\(\Sigma_3\). Thus it is enough to count the nonsingular
matrices in \(\Sigma_3\).
Every \(3\times3\) sign matrix can be normalized by row and column sign
changes to the form
\[
    X=
    \begin{pmatrix}
        1&1&1\\
        1&a&b\\
        1&c&d
    \end{pmatrix},
    \qquad a,b,c,d\in\{\pm1\}.
\]
Row and column sign changes preserve singularity. For the normalized
matrix,
\[
    \det (X)=ad-bc-d+b+c-a.
\]
We now count when this is nonzero.  According to the four possible values
of \((a,b)\), we get
\[
\begin{array}{c|c|c}
(a,b) & \det(X) & \text{number of nonzero choices of }(c,d)\\
\hline
(1,1) & 0 & 0\\
(1,-1) & 2c-2 & 2\\
(-1,1) & 2-2d & 2\\
(-1,-1) & 2c-2d & 2.
\end{array}
\]
Thus exactly
\(
    0+2+2+2=6
\)
of the \(16\) normalized matrices are nonsingular.  
There are \(2^3\) choices of row signs and \(2^3\) choices of column signs.
However, simultaneously reversing all row signs and all column signs does
not change the resulting matrix. Hence each normalized matrix has exactly
\[
    \frac{2^3\cdot2^3}{2}
    =
    2^{3+3-1}
    =
    32
\]
distinct row-and-column sign variants.
Hence the number of nonsingular
\(3\times3\) sign matrices is
\[
    6\cdot32=192.
\]
Therefore, for every fixed sign matrix \(S\), the determinant
\(\det(S\circ X)\) is nonzero for exactly \(192\) choices of
\(X\in\Sigma_3\).
\end{proof}

\subsection{Block amplification}

We now use Lemmas~\ref{lem:per-nonzero} and~\ref{lem:det-192} on several
independent \(3\times3\) blocks. The idea is to test a hypothetical expression for the permanent as a linear
combination of determinants obtained by changing signs of entries.

If \(B_1,\ldots,B_r\) are matrices, we write
\(
    B_1\oplus\cdots\oplus B_r
\)
for the block diagonal matrix with diagonal blocks
\(B_1,\ldots,B_r\) and all off-diagonal blocks equal to zero.
Let
\[
    B=B_1\oplus B_2\oplus\cdots\oplus B_r
    =
    \begin{pmatrix}
        B_1 & 0 & \cdots & 0\\
        0 & B_2 & \cdots & 0\\
        \vdots & \vdots & \ddots & \vdots\\
        0 & 0 & \cdots & B_r
    \end{pmatrix},
\]
where each \(B_j\) is a \(3\times3\) matrix whose entries are distinct
variables, and the variables appearing in different blocks are also
distinct. The following theorem amplifies the \(3\times3\) permanent and determinant
counts to obtain an exponential lower bound.
\begin{theorem}
\label{thm:block-amplification}
Let \(B=B_1\oplus\cdots\oplus B_r\) be as above.  If
\begin{equation}
\label{eq:block-identity}
    \per(B)=\sum_{i=1}^k c_i\det(S_i\circ B)
\end{equation}
where \(c_i\in\mathbb R\) and
\(S_i\in\{\pm1\}^{3r\times3r}\), then
\(
    k\ge \left(\frac83\right)^r.
\)
\end{theorem}

\begin{proof}
We evaluate \eqref{eq:block-identity}
on block diagonal matrices
\(X=X_1\oplus\cdots\oplus X_r\), where \(X_j\in\Sigma_3\).  There are
\(|\Sigma_3|^r=512^r\) choices of \(X\).
For every such \(X\), the left-hand side of
\eqref{eq:block-identity} is nonzero, since \[
    \per(X)=\prod_{j=1}^r \per(X_j),
\]
and each factor is nonzero by Lemma~\ref{lem:per-nonzero}.
Now fix one determinant term \(\det(S_i\circ B)\).  
On a block diagonal matrix \(X=X_1\oplus\cdots\oplus X_r\), this determinant factors as
\[
    \det(S_i\circ X)
    =
    \prod_{j=1}^r \det(S_{i,j}\circ X_j),
\]
where \(S_{i,j}\) is the restriction of \(S_i\) to the \(j\)-th
\(3\times3\) diagonal block.  By Lemma~\ref{lem:det-192}, for each fixed
\(j\), the factor \(\det(S_{i,j}\circ X_j)\) is nonzero for exactly
\(192\) choices of \(X_j\in\Sigma_3\). Therefore
\(\det(S_i\circ X)\) is nonzero for exactly \(192^r\) of the \(512^r\)
choices of \(X\). 

Since the left-hand side of \eqref{eq:block-identity} is nonzero for every
choice of \(X\), at least one determinant term on the right-hand side must
be nonzero. Thus, for every such \(X\),
\[
    1\le
    \left|
    \left\{
        i\in[k]:
        \det(S_i\circ X)\neq0
    \right\}
    \right|.
\]
Summing over all \(512^r\) choices of \(X\), and using the fact that each
fixed determinant term is nonzero for exactly \(192^r\) choices of \(X\),
we obtain
\[
    512^r
    \le
    \sum_X
    \left|
    \left\{
        i\in[k]:
        \det(S_i\circ X)\neq0
    \right\}
    \right|
    =
    k\,192^r.
\]
Therefore
\[
    k\ge
    \left(\frac{512}{192}\right)^r
    =
    \left(\frac83\right)^r.
    \qedhere
\]
\end{proof}

\begin{remark}
The choice of \(3\times3\) blocks is deliberate. For larger blocks, the
same argument compares the number of sign matrices with nonzero permanent
to the number with nonzero signed determinant. The \(3\times3\) count gives
the ratio \(8/3\), and the analogous \(4\times4\) and \(5\times5\) counts
give smaller exponential rates per matrix size.
\end{remark}

\begin{proof}[Proof of Theorem~\ref{thm:permanent-signed-det}]
Let \(A=(a_{ij})\) be an \(n\times n\) matrix of distinct variables. For each \(i\), write \(A_i=S_i\circ A\), where
\(S_i\in\{\pm1\}^{n\times n}\). Then 
\begin{equation}
\label{eq:full-signed-det}
    \per(A)=\sum_{i=1}^k c_i\det(S_i\circ A),
\end{equation}
where \(c_i\in\mathbb R\). Set \(r:=\lfloor n/3\rfloor\) and specialize \(A\) to \(B\oplus I_{n-3r}\),
where \(B:=B_1\oplus\cdots\oplus B_r\), and each \(B_j\) is a \(3\times3\)
matrix of distinct variables. Since
the permanent of a block diagonal matrix is the product of the permanents
of its blocks and \(\per(I_{n-3r})=1\), the left-hand side of \eqref{eq:full-signed-det} specializes to
\(
    \per(B).
\)
For each \(i\), let \(S_i'\) be the restriction of \(S_i\) to the
\(3r\times 3r\) block corresponding to \(B\).  If \(n-3r=0\), set
\(\delta_i=1\).  Otherwise, let \(T_i\) be the restriction of \(S_i\) to
the block corresponding to \(I_{n-3r}\), and set
\(
    \delta_i:=\det(T_i\circ I_{n-3r})\in\{\pm1\}.
\)
Then
\[
    \det(S_i\circ (B\oplus I_{n-3r}))
    =
    \delta_i\det(S_i'\circ B).
\]
Thus \eqref{eq:full-signed-det} specializes to
\[
    \per(B)
    =
    \sum_{i=1}^k c_i\delta_i\det(S_i'\circ B).
\]
By Theorem~\ref{thm:block-amplification},
\[
    k\ge \left(\frac83\right)^r.
\]
Since \(k\) is an integer, the result follows.
\end{proof}

\section{Graph-Theoretic Consequences}\label{sec:consequences}

We now apply the matrix lower bound to obtain lower bounds for Pfaffian
numbers of graphs.

\subsection{Disjoint \(K_{3,3}\) blocks and genus}
We 
write \(G^{\sqcup r}\) for the disjoint
union of \(r\) copies of \(G\).  We first apply the argument to
disjoint copies of \(K_{3,3}\).

\begin{proposition}
\label{prop:disjoint-blocks}
For every \(r\ge1\),
\[
    \pf(K_{3,3}^{\sqcup r})
    \ge
    \left\lceil\left(\frac83\right)^r\right\rceil.
\] 
\end{proposition}

\begin{proof}
Let
\(
    G_r=K_{3,3}^{\sqcup r}.
\)
The weighted biadjacency matrix of \(G_r\) is
\[
    B=B_1\oplus\cdots\oplus B_r,
\]
where \(B_i\) denotes the weighted biadjacency matrix of the \(i\)-th copy
of \(K_{3,3}\).  Thus each \(B_i\) is a \(3\times3\) matrix of distinct
edge variables, and different copies of \(K_{3,3}\) contribute disjoint
sets of variables.
Let \(k=\pf(G_r)\). By
Lemma~\ref{lem:bipartite-signed-determinants}, there are sign
matrices \(S_1,\ldots,S_k \in \{\pm 1\}^{3r\times 3r}\) and real coefficients \(c_1,\ldots,c_k\) such
that
\[
    \per(B)=\sum_{i=1}^k c_i\det(S_i\circ B).
\]
By Theorem~\ref{thm:block-amplification},
\[
    k\ge \left(\frac83\right)^r.
\]
Since \(k\) is an integer, the result follows.
\end{proof}

The disjoint \(K_{3,3}\)-block construction gives a genus lower bound.
We use the standard facts that
\(\gamma(K_{3,3})=1\), and that orientable genus is additive over connected
components:
\[
    \gamma(G_1\sqcup\cdots\sqcup G_r)
    =
    \sum_{i=1}^r \gamma(G_i),
\]
see \cite{BattleHararyKodamaYoungs1962}.
Recall that \(\pfmax(g)\) denotes the maximum Pfaffian number among all
graphs of orientable genus at most \(g\).

\begin{proof}[Proof of Theorem~\ref{thm:intro-genus-main}]
Let
\[
    G=K_{3,3}^{\sqcup g}.
\]
By Proposition~\ref{prop:disjoint-blocks},
\[
    \pf(G)\ge
    \left\lceil\left(\frac83\right)^g\right\rceil .
\]
Since \(\gamma(K_{3,3})=1\), we have \(\gamma(G)=g\). Therefore
\[
    \pfmax(g)\ge\pf(G)\ge
    \left(\frac83\right)^g .
    \qedhere
\]
\end{proof}

\subsection{Matching-covered graphs of genus \(g\)}
We now strengthen the preceding result by showing that the same lower
bound can be witnessed by cubic bipartite matching-covered graphs.
Although \(K_{3,3}^{\sqcup r}\) is disconnected, and hence not
matching-covered, it serves as the basic obstruction underlying the
construction. To transfer its lower bound to more structured graphs, we
will use invariance of the Pfaffian number under bisubdivision.
A graph \(H\) is a \emph{bisubdivision} of a graph \(G\) if \(H\) is
obtained from \(G\) by replacing edges with paths of odd length.
Junchaya, Miranda, and Lucchesi show that bisubdivision preserves the
Pfaffian number for matching-covered graphs
\cite{JunchayaMirandaLucchesi2026}. We will use the following extension
to disjoint unions.

\begin{lemma}\label{lem:bisubdivision}
Let \(G=G_1\sqcup\cdots\sqcup G_r\), where each \(G_i\) is
matching-covered, and let \(H\) be a bisubdivision of \(G\). Then
\[
    \pf(H)=\pf(G).
\]
\end{lemma}

\begin{proof}
Write \(H=H_1\sqcup\cdots\sqcup H_r\), where \(H_\ell\) is the
corresponding bisubdivision of \(G_\ell\), and order the vertices
component by component. There is a bijection between the perfect
matchings of \(G\) and \(H\). Indeed, when an edge is replaced by an odd
path, the matching on the internal vertices is uniquely determined
according to whether or not the original edge belongs to the matching.

The proof of bisubdivision invariance in
\cite{JunchayaMirandaLucchesi2026} shows that, for every orientation
\(D^\ell\) of \(G_\ell\), there is an orientation
\(\widetilde D^\ell\) of \(H_\ell\) and a sign
\(\varepsilon^\ell\in\{\pm1\}\) such that corresponding perfect
matchings have signs differing by the fixed factor
\(\varepsilon^\ell\).
Let \(k=\pf(G)\), and write
\[
    \PM(G)=\sum_{i=1}^k c_i\Pf(A_{D_i}).
\]
For each \(i\), apply the preceding correspondence to the restriction of
\(D_i\) to each component \(G_\ell\), and combine the resulting
orientations to obtain an orientation \(\widetilde D_i\) of \(H\).
Since the vertices are ordered component by component, the matching
signs multiply over the components. Thus, for some
\(\varepsilon_i\in\{\pm1\}\), corresponding perfect matchings \(M\) of
\(G\) and \(\widetilde M\) of \(H\) satisfy
\[
    \sign_{\widetilde D_i}(\widetilde M)
    =
    \varepsilon_i\sign_{D_i}(M).
\]
Hence
\[
    \PM(H)
    =
    \sum_{i=1}^k c_i\varepsilon_i\Pf(A_{\widetilde D_i}),
\]
so \(\pf(H)\le\pf(G)\). 
For the reverse inequality, start with a minimum Pfaffian representation
of \(H\). Applying the reverse direction of the same argument to each
orientation, component by component, gives a Pfaffian representation of
\(G\) with the same number of terms, up to fixed sign changes in the
coefficients. Hence
\(
    \pf(G)\le\pf(H).
\) Therefore
\[
    \pf(H)=\pf(G). \qedhere
\]
\end{proof}

The disjoint-block lower bound extends immediately to any graph containing
such blocks as a conformal subgraph.

\begin{corollary}
\label{cor:conformal-k33}
Let \(G\) be a graph. Suppose \(G\) contains a conformal subgraph \(H\)
with exactly \(r\) connected components, each of which is a
bisubdivision of \(K_{3,3}\). Then
\[
    \pf(G)\ge
    \left\lceil\left(\frac83\right)^r\right\rceil .
\]
\end{corollary}

\begin{proof}
By Proposition~\ref{prop:conformal}, Lemma~\ref{lem:bisubdivision}, and
Proposition~\ref{prop:disjoint-blocks},
\[
    \pf(G)\ge\pf(H)
    =\pf(K_{3,3}^{\sqcup r})
    \ge
    \left\lceil\left(\frac83\right)^r\right\rceil .\qedhere
\]
\end{proof}

This shows that several disjoint \(K_{3,3}\) obstructions inside a
conformal subgraph force the Pfaffian number to grow exponentially.
We now show that this remains true even when these
obstructions are combined into a single matching-covered graph.
In fact, the resulting graph can be chosen to be bipartite and cubic.
A graph is \emph{cubic} if every vertex has degree \(3\).

\begin{theorem}
\label{thm:cubic-matching-covered}
For every \(g\ge1\), there exists a cubic bipartite
matching-covered graph \(G_g\) of orientable genus exactly \(g\) such that
\[
    \pf(G_g)\ge
    \left\lceil\left(\frac83\right)^g\right\rceil .
\]
\end{theorem}

\begin{proof}
For \(g=1\), take \(G_1=K_{3,3}\). Now assume \(g\ge2\). Take \(g\) copies
\(C_1,\ldots,C_g\) of \(K_{3,3}\). Write the bipartition of \(C_i\) as
\(
    U_i=\{a_i,b_i,c_i\},
\) and  \(
    V_i=\{\alpha_i,\beta_i,\gamma_i\}.
\)
In each copy \(C_i\), replace the edge \(a_i\alpha_i\) by the path
\[
    a_i-x_i-y_i-\alpha_i .
\]
Then add the edges
\[
    x_i y_{i+1},
    \qquad i=1,\ldots,g,
\]
where indices are taken modulo \(g\). We call these the \emph{joining edges}. The construction is shown in
Figure~\ref{fig:cubic-construction} for \(g=3\).
The graph \(G_g\) is connected and bipartite, with \(x_i\) in the same
bipartition class as \(\alpha_i\) and \(y_i\) in the same bipartition class
as \(a_i\). Since every vertex has degree \(3\), the graph \(G_g\) is cubic.

We now show that \(G_g\) is matching-covered. Let \(B_i\) be the subdivided
copy of \(K_{3,3}\) obtained from \(C_i\) by replacing \(a_i\alpha_i\) with
the path \(a_i-x_i-y_i-\alpha_i\). Each \(B_i\) is matching-covered. A perfect matching of \(K_{3,3}\) either lifts by replacing \(a_i\alpha_i\)
with the two edges \(a_ix_i\) and \(y_i\alpha_i\), or lifts by adding the
edge \(x_iy_i\). Hence every edge inside some \(B_i\) belongs to a perfect
matching of \(B_i\), and this matching extends to a perfect matching of
\(G_g\) by choosing perfect matchings in the other blocks and using no
joining edges.
It remains to cover the joining edges. The set
\[
    \{x_jy_{j+1}: j=1,\ldots,g\}
    \cup
    \{a_j\beta_j,\ b_j\alpha_j,\ c_j\gamma_j: j=1,\ldots,g\}
\]
is a perfect matching of \(G_g\), and it contains all joining edges. Thus
every edge of \(G_g\) belongs to a perfect matching. Since \(G_g\) is
connected, it is matching-covered.

\begin{figure}[ht!]
\centering
\resizebox{0.48\textwidth}{!}{
\begin{tikzpicture}[
    every node/.style={circle, draw, inner sep=1.15pt, font=\scriptsize},
    blocklabel/.style={draw=none, circle, font=\small},
    internal/.style={gray!55, line width=0.32pt},
    path/.style={black, line width=0.6pt},
    connector/.style={blue, line width=0.9pt}
]

\node (al1) at (-0.75,2.75) {$\alpha_1$};
\node (be1) at (-0.75,2.15) {$\beta_1$};
\node (ga1) at (-0.75,1.55) {$\gamma_1$};

\node (a1) at (0.75,2.75) {$a_1$};
\node (b1) at (0.75,2.15) {$b_1$};
\node (c1) at (0.75,1.55) {$c_1$};

\foreach \u in {a,b,c} {
    \foreach \v in {al,be,ga} {
        \draw[internal] (\u1) -- (\v1);
    }
}

\node (y1) at (-0.42,3.35) {$y_1$};
\node (x1) at (0.42,3.35) {$x_1$};

\draw[white, line width=3pt] (a1) -- (al1);
\draw[path] (a1) -- (x1) -- (y1) -- (al1);

\node[blocklabel] at (0,1.05) {$C_1\cong K_{3,3}$};

\foreach \i/\X/\Y in {2/3.05/-1.35,3/-3.05/-1.35} {

    \node (a\i) at (\X-0.75,\Y+0.6) {$a_{\i}$};
    \node (b\i) at (\X-0.75,\Y) {$b_{\i}$};
    \node (c\i) at (\X-0.75,\Y-0.6) {$c_{\i}$};

    \node (al\i) at (\X+0.75,\Y+0.6) {$\alpha_{\i}$};
    \node (be\i) at (\X+0.75,\Y) {$\beta_{\i}$};
    \node (ga\i) at (\X+0.75,\Y-0.6) {$\gamma_{\i}$};

    \foreach \u in {a,b,c} {
        \foreach \v in {al,be,ga} {
            \draw[internal] (\u\i) -- (\v\i);
        }
    }

    \node (x\i) at (\X-0.42,\Y+1.22) {$x_{\i}$};
    \node (y\i) at (\X+0.42,\Y+1.22) {$y_{\i}$};

    \draw[white, line width=3pt] (a\i) -- (al\i);
    \draw[path] (a\i) -- (x\i) -- (y\i) -- (al\i);

    \node[blocklabel] at (\X,\Y-1.05) {$C_{\i}\cong K_{3,3}$};
}

\draw[connector] (x1) to[out=-25,in=120,looseness=0.85] (y2);
\draw[connector] (x2) to[out=-165,in=-15,looseness=0.45] (y3);
\draw[connector] (x3) to[out=60,in=205,looseness=0.85] (y1);

\end{tikzpicture}
}
\caption{The graph \(G_3\).}
\label{fig:cubic-construction}
\end{figure}
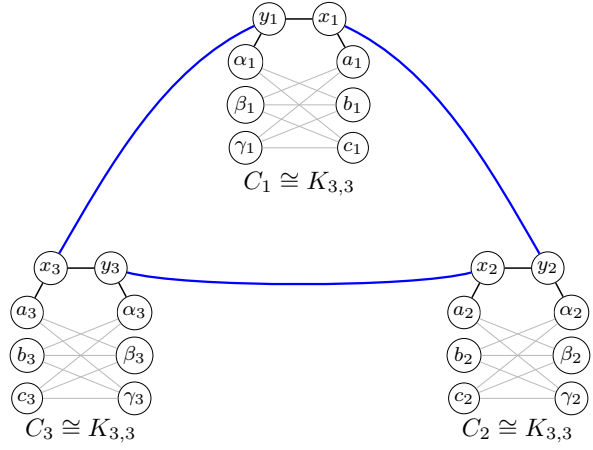

Let \(H_g\) be the spanning subgraph obtained from \(G_g\) by deleting the
joining edges \(x_i y_{i+1}\). Then \(H_g\) is a bisubdivision of
\(K_{3,3}^{\sqcup g}\). By Lemma~\ref{lem:bisubdivision},
\[
    \pf(H_g)=\pf(K_{3,3}^{\sqcup g}).
\]
Since \(H_g\) is a spanning subgraph of \(G_g\) and has a perfect
matching, it is conformal.
Proposition~\ref{prop:conformal} therefore gives
\[
    \pf(G_g)\ge \pf(H_g).
\]
Together with Proposition~\ref{prop:disjoint-blocks}, this gives
\[
    \pf(G_g)
    \ge
    \left\lceil\left(\frac83\right)^g\right\rceil .
\]
Finally, we show that \(G_g\) has orientable genus exactly \(g\).
Each component of \(H_g\) is a bisubdivision of \(K_{3,3}\), and
subdividing edges does not change the orientable genus. Since
\(\gamma(K_{3,3})=1\) and orientable genus is additive over connected
components, we have
\(
    \gamma(H_g)=g.
\)
Since \(H_g\subseteq G_g\), genus monotonicity gives
\(
    \gamma(G_g)\ge g.
\)
For the reverse inequality, take an orientable surface with \(g\) handles
arranged cyclically. Embed the \(i\)-th copy of \(K_{3,3}\) on the
\(i\)-th handle so that the edge \(a_i\alpha_i\) is incident with
the region between the handles, and then subdivide this edge as
\(a_i-x_i-y_i-\alpha_i\). Thus \(x_i\) and \(y_i\) are accessible from
this region. The joining edges \(x_i y_{i+1}\) can then be drawn in this
region, in cyclic order, without crossings. Hence
\(
    \gamma(G_g)\le g.
\)
Therefore
\[
    \gamma(G_g)=g.
    \qedhere
\]
\end{proof}

\begin{remark}
One can also obtain non-bipartite matching-covered examples. For \(g\ge2\), start with the
graph \(G_g\) from Theorem~\ref{thm:cubic-matching-covered} and add the two
edges
\(
    a_1b_1 \text{ and } \alpha_1\beta_1 .
\)
Let the resulting graph be \(G_g^+\). Then \(G_g^+\) is non-bipartite,
because it contains the triangle
\(
    a_1b_1\beta_1a_1 .
\)
The graph \(G_g^+\) is still matching-covered. Indeed, all old edges already
belong to perfect matchings in \(G_g\) and the two new edges belong to the
perfect matching containing
\[
    a_1b_1,\quad
    \alpha_1\beta_1,\quad
    c_1\gamma_1,\quad
    x_1y_1,
\]
together with perfect matchings of all remaining blocks. Moreover, \(G_g^+\) also has genus \(g\). The genus cannot decrease because
\(G_g\subseteq G_g^+\), and it does not increase because the two new edges
can be drawn on the first handle. Finally, since \(G_g\) is a spanning subgraph of \(G_g^+\) and has a
perfect matching, it is conformal. Hence Proposition~\ref{prop:conformal} gives
\[
    \pf(G_g^+)\ge \pf(G_g)
    \ge
    \left\lceil\left(\frac83\right)^g\right\rceil .
\]
\end{remark}

\subsection{Complete bipartite and complete graphs}
Next we prove the lower bound for the Pfaffian number of complete bipartite graphs.

\begin{proof}[Proof of Theorem~\ref{thm:intro-complete-bipartite}]
Let \(A\) be the weighted biadjacency matrix of \(K_{n,n}\). Thus \(A\)
is an \(n\times n\) matrix whose entries are distinct variables, and
\[
    \PM(K_{n,n})=\per(A).
\]
Let \(k=\pf(K_{n,n})\). By
Lemma~\ref{lem:bipartite-signed-determinants}, there exist sign
matrices \(S_1,\ldots,S_k\in\{\pm1\}^{n\times n}\) and real coefficients
\(c_1,\ldots,c_k\) such that
\[
    \per(A)=\sum_{i=1}^k c_i\det(S_i\circ A).
\]
By Theorem~\ref{thm:permanent-signed-det},
\[
    k\ge
    \left\lceil
    \left(\frac83\right)^{\lfloor n/3\rfloor}
    \right\rceil . \qedhere
\]
\end{proof}

The complete graph bound follows because \(K_{2n}\) contains a spanning
copy of \(K_{n,n}\).

\begin{proof}[Proof of Theorem~\ref{thm:intro-complete-graphs}]
Split the \(2n\) vertices of \(K_{2n}\) into two parts of size \(n\).
The edges between the two parts form a spanning subgraph isomorphic to
\(K_{n,n}\). Since \(K_{n,n}\) has a perfect matching, it is conformal. Therefore, by
Proposition~\ref{prop:conformal},
\[
    \pf(K_{2n})\ge \pf(K_{n,n}).
\]
The result now follows from
Theorem~\ref{thm:intro-complete-bipartite}.
\end{proof}

In particular, for every even \(n\), the complete graph itself witnesses
the corresponding \(n\)-vertex lower bound.

\begin{corollary}
\label{cor:vertex-lower-bound}
For every even integer \(n\ge6\),
\[
    \pf(K_n)
    \ge
    \left\lceil
    \left(\frac83\right)^{\lfloor n/6\rfloor}
    \right\rceil .
\]
\end{corollary}

\begin{proof}
Write \(n=2m\). By Theorem~\ref{thm:intro-complete-graphs},
\[
    \pf(K_n)
    =
    \pf(K_{2m})
    \ge
    \left\lceil
    \left(\frac83\right)^{\lfloor m/3\rfloor}
    \right\rceil
    =
    \left\lceil
    \left(\frac83\right)^{\lfloor n/6\rfloor}
    \right\rceil . \qedhere
\]
\end{proof}

\section{Discussion}

We proved that Pfaffian numbers can grow exponentially with orientable genus.
Together with the Galluccio--Loebl--Tesler upper bound, this gives
\[
    \left(\frac83\right)^g
    \le
    \pfmax(g)
    \le
    4^g .
\]        
The examples can be chosen cubic, bipartite, and matching-covered.
We also obtain exponential lower bounds for complete bipartite graphs and even
complete graphs.
The remaining gap is to determine the correct exponential growth rate of
\(\pfmax(g)\). Equivalently, one would like to understand
\[
    \limsup_{g\to\infty} \pfmax(g)^{1/g}.
\]
Our results place this quantity between \(8/3\) and \(4\). It remains open
to narrow this gap and determine the correct exponential growth rate.
Another direction is to understand the Pfaffian numbers of specific graph families more precisely.
Our bounds give exponential lower bounds for both, but their exact growth
rates are not yet known.

\bibliographystyle{plainurl}
\bibliography{references}

\end{document}